\documentclass[isup,preprint]{imsart_isup}

\RequirePackage[OT1]{fontenc}
\RequirePackage{amsthm,amsmath,amsfonts,bbm}
\RequirePackage[numbers]{natbib}
\RequirePackage[colorlinks,citecolor=blue,urlcolor=blue,driverfallback=dvipdfm]{hyperref}
\RequirePackage{hypernat}
\RequirePackage[french,english]{babel}

\usepackage[latin1]{inputenc}

\startlocaldefs
\numberwithin{equation}{section}
\theoremstyle{plain}
\newtheorem{thm}{Theorem}[section]
\endlocaldefs

\begin{document}

\begin{frontmatter}
\title{Classification under local differential privacy}
\runtitle{Classification under local differential privacy}

\begin{aug}
\author{\fnms{Thomas} \snm{Berrett}\thanksref{t1,t2,m1}\ead[label=e1]{thomas.berrett@ensae.fr}}
\and
\author{\fnms{Cristina} \snm{Butucea}\thanksref{t1,m1}\ead[label=e2]{cristina.butucea@ensae.fr}}

\thankstext{t1}{Financial support from the French National Research Agency (ANR) under the grant Labex Ecodec (ANR-11-LABEX-0047)}
\thankstext{t2}{Financial support from the French National Research Agency (ANR) under the grant ANR-17-CE40-0003 HIDITSA}
\runauthor{T. Berrett and C. Butucea}

\affiliation{CREST, ENSAE, Institut Polytechnique de Paris\thanksmark{m1} }

\address{5, avenue Henry Le Chatelier, \\
F-91120 Palaiseau, France\\
\printead{e1}\\
\phantom{E-mail:\ }\printead*{e2}}

\end{aug}

\begin{abstract}
We consider the binary classification problem in a setup that preserves the privacy of the original sample. We provide a privacy mechanism that is $\alpha-$locally differentially private and then construct a classifier based on the private sample that is universally consistent in Euclidean spaces. Under stronger assumptions, we establish the minimax rates of convergence of the excess risk and see that they are slower than in the case when the original sample is available.
\end{abstract}

\begin{keyword}[class=AMS]
\kwd[Primary ]{62G08}
\kwd[; secondary ]{62H30}
\end{keyword}

\begin{keyword}
\kwd{classification}
\kwd{local differential privacy}
\kwd{minimax rates}
\kwd{universal consistency}
\end{keyword}

\end{frontmatter}

\section{Introduction}

One of the most frequently studied problems in machine learning and statistics is to make predictions of a binary outcome $Y$ in $\{0,1\}$ given the input variable $X$ in $\mathbb{R}^d$, typically based on an independent and identically distributed sample from the distribution of $(X,Y)$; see, for example, \cite{PTPR:96} for an overview of the area. Despite the long history of this problem there are still many open problems and it remains an active topic of research. Recent work has focused on weakening commonly-made assumptions \cite{CBS2019}, studying situations in which the training data comes from a different distribution to the test data \cite{CFS2019,CW2019}, and making predictions under constraints on allowable classifiers \cite{ZWSP2013}.

In recent years, it has become clear that in certain studies there is a need to preserve the privacy of the individuals whose data is collected. As a way of formalising the problem, the framework of differential privacy, see \cite{Dwork2006} and \cite{Dwork2008}, has prevailed as a natural solution. The privacy of the individuals is protected by randomizing the original data to produce the random variables $Z_1,...,Z_n$ defined on some measurable space $(\mathcal{Z}^n, \mathcal{B}^n)$, and the following statistical analysis is based soley on $Z_1,\ldots,Z_n$ without knowledge of the original data. The randomization is performed through a Markov kernel, also known as privacy mechanism, $Q : \mathcal{B}^n \times (\mathbb{R}^d \times \{0,1\})^{\otimes n} \to [0,1]$. A privacy mechanism is called $\alpha-$differentially private, for $\alpha >0$, if
\begin{equation}
\label{Eq:Privacy}
\sup_{A \in \mathcal{B}^n} \sup_{u, u': d(u,u')=1} \frac{Q(A|u)}{Q(A|u')} \leq e^\alpha,
\end{equation}
where $d: (\mathbb{R}^d \times \{0,1\})^{\otimes n} \times (\mathbb{R}^d \times \{0,1\})^{\otimes n} \rightarrow \{0,1,2,\ldots\}$ is the Hamming distance, counting the number of positions $i \in \{1,\ldots,n\}$ such that $u_i \neq u_i'$.

 In the global, also called central, model for differential privacy, the whole sample is trusted to one person/machine in order to produce the randomized sample. However, in certain situations this is not feasible. In the setup of \emph{local} differential privacy that we consider here, the sample is randomized one at a time by different persons/machines so that the $i$th original data point $(X_i,Y_i)$ only needs to be seen by the $i$th individual. Formally, we assume the sequential structure
\[
	\{X_i,Y_i,Z_1,\ldots,Z_{i-1}\} \rightarrow Z_i \quad \text{and} \quad Z_i \perp \!\!\! \perp (X_j,Y_j) | \{X_i,Y_i,Z_1,\ldots,Z_{i-1}\} \text{ for } j \neq i.
\]
With this structure, the full conditional distribution $Q$ can be expressed in terms of the conditional distributions $Q_i: \mathcal{B} \times (\mathbb{R}^d \times \{0,1\}) \times \mathcal{Z}^{i-1} \rightarrow [0,1]$, where
\[
	Z_i | \{(X_i,Y_i)=u_i,Z_1=z_1,\ldots,Z_{i-1}=z_{i-1}\} \sim Q_i( \cdot | u_i ,z_1,\ldots,z_{i-1})
\]
for $i=1,\ldots,n$. The privacy constraint~\eqref{Eq:Privacy} then reduces to
\[
	\sup_{A \in \mathcal{B}} \sup_{z_1,\ldots,z_{i-1}} \sup_{u_i,u_i'} \frac{Q_i(A|u_i,z_1,\ldots,z_{i-1})}{Q_i(A|u_i',z_1,\ldots,z_{i-1})} \leq e^\alpha,
\]
for each $i=1,\ldots,n$; if this holds then the privacy mechanism $Q$ is said to be $\alpha-$locally differentially private ($\alpha-$LDP). Let $\mathcal{Q}_\alpha$ denote the set of $\alpha-$LDP Markov kernels.

Many local privacy mechanisms of interest are of a simpler, non-interactive, form in which $Q_i( \cdot | u_i,z_1,\ldots,z_{i-1})$ does not depend on $z_1,\ldots,z_{i-1}$. In particular, a non-interactive privacy mechanism has a product form:
$$
Q(A_1\times ... \times A_n|(u_1,...,u_n)) = \prod_{i=1}^n Q_i(A_i|u_i).
$$
Thus, a non-interactive privacy mechanism satisfies $\alpha-$LDP if and only if, for all $i=1,...,n$,
$$
\sup_{A_i \in \mathcal{B}} \sup_{u_i, u'_i} \frac{Q_i(A_i|u_i)}{Q_i(A_i|u'_i)} \leq e^\alpha.
$$

Due to more recent papers such as \cite{DJW13}, the attention of the statistical community has been attracted to this field. In~\cite{DJW13} it is shown, for example, that the minimax rate of convergence for estimating the mean of a $d$ dimensional vector is slower in the context of privatized data as opposed to the classical problem and information theoretic bounds are established that allow the construction of lower bound results in many other estimation problems. The same authors also considered in \cite{DJW18} nonparametric density estimation and showed that the minimax MISE is of order $(n \alpha^2)^{-\frac {\beta}{2\beta +2}}$ over Sobolev smooth functions, instead of the classical rate $n^{-\frac{\beta}{2\beta +1 }}$. More general linear functionals of the probability density have been considered by \cite{RS2019} and they showed that the modulus of continuity with respect to the total variation distance drives the rates in presence of privacy, instead of the Hellinger distance. Adaptation to the smoothness of both the privacy mechanism and of the associated estimator has been considered by \cite{BDKS}. They showed a double elbow effect for the adaptive estimation rates for more general integrated risks over Besov class of functions. The sparse regression model has been considered in the context of local privacy by \cite{WangXu}; they give a nearly optimal procedure in the low-dimensional sparse case and consider the case where only the responses are privatized.
To the best of our knowledge, classification under local privacy constraints has never been considered from the point of view of statistical inference. 

In the sequel, we exhibit a non-interactive $\alpha-$LDP privacy mechanism and construct a classifier that is universally consistent in Euclidean spaces. Under standard assumptions on the probability distribution of the input variables as well as on the regression function (H\"older smoothness $\beta \in [0,1]$ and margin assumption with parameter $\gamma$), similar to the assumptions in \cite{AudibertTsybakov07}, we show that the minimax rate of convergence of the excess risk over all classifiers and all (possibly interactive) $Q \in \mathcal{Q}_\alpha$ is $(n\alpha^2)^{-\frac{\beta (1+\gamma)}{2 \beta + 2d}}$, and that this rate is achieved by our non-interactive approach. This rate is slower than the minimax rate for classification $n^{-\frac{\beta (1+\gamma)}{2 \beta + d}}$ attainable when the original sample is available \cite{AudibertTsybakov07}.

\section{Main results}

Let $(X,Y)$ be a random vector taking values in $\mathbb{R}^d \times \{0,1\}$ and moreover let $(X_1,Y_1),\ldots,(X_n,Y_n)$ be independent and identically distributed copies of $(X,Y)$. Our task is to find a mechanism that outputs locally privatised data $Z_1,\ldots,Z_n$ and then to use this privatised data to construct a classifier $C_n : [0,1]^d \rightarrow \{0,1\}$ to predict $Y$ from $X$.

The performance of this classifier will be measured through its excess risk, defined as follows. If the distribution $P$ of $(X,Y)$ were known, we could calculate $\eta(x):= \mathbb{P}(Y=1 |X=x)$ and use the Bayes classifier $C^\mathrm{B}(x):= \mathbbm{1}_{\{\eta(x) \geq 1/2\}}$. This minimises the risk
\[
	R_P(C):= \mathbb{P}(C(X) \neq Y)
\]
over all classifiers $C$. The excess risk of a (data dependent) classifier $C_n$ is then defined to be
\begin{align*}
	\mathcal{E}_P(C_n) &:= R_P(C_n) - R_P(C^\mathrm{B}) = \mathbb{E} \bigl[ \{ \mathbb{P}(C_n(X)=0|X) - \mathbbm{1}_{\{\eta(X) < 1/2\}} \} \{2 \eta(X) -1 \} \bigr].
\end{align*}

We now introduce the privacy mechanism and the classifier that we will study in this article. For a bandwidth parameter $h \in (0,\infty)$ and $j=(j_1,...,j_d) \in \mathbb{Z}^d$ set the grid points $x_j:= ( j_1h, \ldots, j_dh )$, and for each $x \in [0,1]^d$ define $j^*(x):= \mathrm{argmin}_{j \in \mathbb{Z}^d } \|x-x_j\|$. With a slight abuse of notation we will assume that our sample size is even and is given by $2n$. For each $i=1,\ldots,2n$, individual $i$ should form the binary array
\[
	B_i:= \bigl( \mathbbm{1}_{\{\|X_i-x_j\|_\infty < h\}} \bigr)_{j \in \mathbb{Z}^d}
\]
and generate a random array $\epsilon_i = (\epsilon_{ij})_{ j \in \mathbb{Z}^d}$ with independent and identically distributed Laplace$(\alpha/2^{d+1})$ entries. The $Z_i$ are then given by
\[
       Z_i=B_i + \epsilon_i, \quad \text{if }i \in \{1,\ldots,n\}, \text{  and }
       Z_i = Y_i B_i + \epsilon_i, \quad \text{if }i \in \{n+1,\ldots,2n\}.
\]
Note that this privacy mechanism is local, as each $Z_i$ can be calculated by a different person/machine, and non-interactive, i.e. no other privatized data $Z_j$ with $j \not = i$ are used to encode additional knowledge into $Z_i$.

Our first claim is that this privacy mechanism respects the $\alpha$-LDP constraint.
\begin{thm}
\label{Prop:Private}
Let $Q$ be the privacy mechanism that takes $(X_1,Y_1),\ldots,(X_{2n},Y_{2n})$ as input and outputs $Z_1,\ldots,Z_{2n}$, as described above. Then $Q \in \mathcal{Q}_\alpha$.
\end{thm}

Now let $X=x_0$ be the test point we wish to classify. Calculate
\[
	T_n(x_0):= \frac{1}{n} \sum_{i=n+1}^{2n} Z_{i j^*(x_0)} - \frac{1}{2n} \sum_{i=1}^n Z_{i j^*(x_0)},
\]
and define the classifier $C_n(x_0):= \mathbbm{1}_{\{T_n(x_0) \geq 0\}}$. Our first main result is that this classifier is universally consistent, in the following sense.
\begin{thm}
\label{Thm:UniversalConsistency}
Suppose that $h \rightarrow 0$ and $n \alpha^2 h^{2d} \rightarrow \infty$ as $n \rightarrow \infty$. Then, for any probability distribution $P$ on $\mathbb{R}^d \times \{0,1\}$, we have that $\mathcal{E}_P(C_n) \rightarrow 0$ as $n \rightarrow \infty$.
\end{thm}
In comparison with classical universal consistency results, see for example Theorem 6.2 of \cite{PTPR:96}, we see that we take $n \alpha^2 h^{2d} \rightarrow \infty$ rather than $n h^d \rightarrow \infty$. The change in the exponent of $h$ seems to be an unavoidable consequence of the privacy constraints.

Our remaining results study the minimax rate of convergence of the excess risk of our classifier over certain classes of distributions $P$, and also show that this matches the minimax optimal rate of convergence under an $\alpha$-LDP constraint. For a class $\mathcal{P}$ of distributions of $(X,Y)$, for a privacy parameter $\alpha >0$ and for a sample size $n \in \mathbb{N}$, we will write
\[
	\mathcal{R}_{n,\alpha}(\mathcal{P}):= \inf_{\substack{Q \in \mathcal{Q}_\alpha \\C_n}} \sup_{P \in \mathcal{P}} \mathcal{E}_P(C_n)
\]
for the minimax excess risk, where the infimum is taken over all Markov kernels $Q \in \mathcal{Q}_\alpha$ and all classifiers $C_n$ depending only on the privatised data $Z_1,\ldots,Z_n$.

We now introduce the classes of distributions over which our results will hold, and note that they are very similar to those considered in~\cite{AudibertTsybakov07}. Given $\beta \in (0,1]$ and $L>0$ we will say that the distribution of $(X,Y)$ satisfies the $(\beta,L)$-H\"older smoothness condition if
\[
	|\eta(x') - \eta(x)| \leq L \|x-x'\|^\beta
\]
for all $x,x' \in [0,1]^d$. Given $\alpha \geq 0$ and $C_0 >0$ we will say that the distribution of $(X,Y)$ satisifies the $(\gamma,C_0)$-margin condition if
\[
	\mathbb{P}( 0 < |\eta(X) - 1/2| \leq t) \leq C_0 t^\gamma
\]
for all $t>0$. For $c_0,r_0 > 0 $ we say that a Lebesgue measurable set $A \subset [0,1]^d$ is $(c_0,r_0)$-regular if
\[
	\lambda\{ A \cap B_x(r) \} \geq c_0 \lambda\{B_x(r)\}
\]
for all $r \in (0,r_0]$ and $x \in A$. Then for $c_0,r_0,\mu > 0$ we will finally say that the distribution of $(X,Y)$ satisfies the $(c_0,r_0,\mu)$-strong density assumption if $X$ has a density $f$ such that $\mathrm{supp}(f):=\{x:f(x)>0\} \subseteq [0,1]^d$ is $(c_0,r_0)$-regular and $f(x) \geq \mu $ for all $x \in \mathrm{supp}(f)$. We then write $\theta=(\beta,\gamma, C_0,L,c_0,r_0,\mu)$ and $\mathcal{P}(\theta)$ for the class of all distribution of $(X,Y)$ that satisfy the $(\beta,L)$-H\"older smoothness condition, the $(\gamma,C_0)$-margin condition, the $(c_0,r_0,\mu)$-strong density assumption. Our main theorem concerning minimax rates of convergence is the following.

\begin{thm}
\label{Thm:Main}
Fix $\theta=(\beta,\gamma,C_0,L,c_0,r_0,\mu)$ with $\beta \in (0,1], \gamma \in [0,\infty)$ and $C_0,L,c_0,$ $r_0,\mu \in (0,\infty)$ such that $\beta \gamma \leq d$. Then there exist $c=c(d,\theta)$ and $C=C(d,\theta)$ such that, for all $n \in \mathbb{N}$ and $\alpha \in (0,1]$, we have
\[
	c (n \alpha^2)^{-\frac{\beta(1+\gamma)}{2\beta+2d}} \leq \mathcal{R}_{n,\alpha}(\mathcal{P}(\theta)) \leq C (n \alpha^2)^{-\frac{\beta(1+\gamma)}{2\beta+2d}}
\]
\end{thm}
Note that the rates are slower in the privatized setup than the rates obtained by \cite{AudibertTsybakov07}, $n^{-\frac{\beta(1+\gamma)}{2\beta+d}}$, in the classical setup. Such a loss was already noticed for minimax estimation of a probability density function with privatized data, see e.g. \cite{DJW18}. Theorem~\ref{Thm:Main} is a consequence of Theorems~\ref{Prop:UpperBound} and~\ref{Prop:LowerBound} below.

The following result bounds the excess risk of this classifier and thus establishes the upper bound on $\mathcal{R}_{n,\alpha}(\mathcal{P}(\theta))$ in Theorem~\ref{Thm:Main}. Note that the restriction to the case $\beta \in (0,1]$ is mainly due to the regressogram that defines our classifier. A smoother version should allow to cover cases for $\beta >1$.
\begin{thm}
\label{Prop:UpperBound}
Fix $\theta=(\beta, \gamma,C_0,L,c_0,r_0,\mu)$ with $\beta \in (0,1], \gamma \in [0,\infty)$ and $C_0,L,c_0,$ $r_0,\mu \in (0,\infty)$. Then there exists $A=A(d,\beta,\gamma,c_0,r_0,\mu)>0$ such that
\[
	\sup_{P \in \mathcal{P}(\theta)} \mathcal{E}_P(C_n) \leq A C_0  \biggl\{ \frac{L^d}{(n \alpha^2)^{\beta/2}} \biggr\}^{\frac{1+\gamma}{d + \beta}}.
\]
\end{thm}
The following lower bound complements our previous upper bound and completes the proof of Theorem~\ref{Thm:Main}.
\begin{thm}
\label{Prop:LowerBound}
Fix $\beta \in (0,1], \gamma \in [0,\infty)$ and $C_0,L \in (0,\infty)$ such that $\beta \gamma \leq d$. Then there exists $c=c(d,\beta,\gamma,C_0,L)$ such that, for all $c_0,r_0,\mu \in (0,c)$, $n \in \mathbb{N}$ and $\alpha \in (0,1]$ we have
\[
	\mathcal{R}_{n, \alpha}(\mathcal{P}(\beta,\gamma,C_0,L,\mu)) \geq c (n \alpha^2)^{-\frac{\beta(1+\gamma)}{2\beta+2d}}.
\]
\end{thm}

\section{Proofs}

\begin{proof}[Proof of Theorem~\ref{Prop:Private}]
First consider $i \in \{1,\ldots,n\}$. Any possible value of $X_i$ produces a value $B_i$ that is a binary array with at most $2^d$ non-zero entries; let $b$ and $b'$ be any two such possible values. Then, writing $f_{Z|X}^{(1)}(z|x)$ for the conditional density of $Z_i$ given $X_i$, for any $z \in \mathbb{R}^{\mathbb{Z}^d}$ we have
\begin{align*}
	\frac{f_{Z|X}(z|x)}{f_{Z|X}(z|x')} = \prod_{j \in \mathbb{Z}^d} \frac{\frac{\alpha}{2^{d+2}} e^{-\alpha |z_j - b_j|/2^{d+1}}}{\frac{\alpha}{2^{d+2}} e^{-\alpha |z_j - b_j'|/2^{d+1}}} &= \exp \biggl( \frac{\alpha}{2^{d+1}} \sum_{j \in \mathbb{Z}^d} (|z_j - b_j'| - |z_j - b_j|) \biggr) \\
	&\leq \exp\biggl( \frac{\alpha}{2^{d+1}} \sum_{j \in \mathbb{Z}^d} |b_j - b_j'|  \biggr) \leq e^\alpha.
\end{align*}
Since $(X_i,Y_i) \perp \!\!\! \perp Z_i | X_i$, this proves that $\alpha$-LDP is respected for the first half of the individuals. When $i \in \{n+1,\ldots,2n\}$ we have $(X_i, Y_i) \perp \!\!\! \perp Z_i | (X_i,Y_i)$. Since $Y_iB_i$ is always a binary array with at most $2^d$ non-zero entries, the same calculation as above shows that $\alpha$-LDP is also respected for the second half of the individuals. This concludes the proof.
\end{proof}
\begin{proof}[Proof of Theorem~\ref{Thm:UniversalConsistency}]
As shorthand, when $X=x_0$ and for $i \in \{1,\ldots,n\}$, write $W_i = W_i(x_0):=Z_{i+n,j^*(x_0)}$, $U_i=U_i(x_0):= Z_{i j^*(x_0)}$ and
\[
	T_n=T_n(x_0) = \frac{1}{n} \sum_{i=1}^n \Bigl( W_i - \frac{1}{2} U_i \Bigr).
\]
Using the fact that the moment generating function of a Laplace$(\alpha/2^{d+1})$ random variable is given by $\lambda \mapsto 1/(1-2^{2d+2}\lambda^2/\alpha^2)$ on $(- \alpha/2^{d+1}, \alpha/2^{d+1})$, for $\lambda \in (-n\alpha/2^{d+1},n\alpha/2^{d+1})$ we have that
\begin{align*}
	\mathbb{E} \bigl\{ e^{\lambda(T_n - \mathbb{E} T_n)} \bigr\} =\bigl[ \mathbb{E} \bigl\{ e^{(\lambda/n)(W-U/2 - \mathbb{E}W - \mathbb{E}U/2)} \bigr\} \bigr]^n \leq \biggl\{ \frac{e^{(1/8+1/32)\lambda^2/n^2}}{(1-\frac{2^{2d+2} \lambda^2}{n^2 \alpha^2}) (1-\frac{ 2^{2d} \lambda^2}{n^2 \alpha^2})} \biggr\}^n.
\end{align*}
It therefore follows that, for $t \in (0,1]$,
\begin{align}
\label{Eq:Subgaussian}
	\mathbb{P}( T_n &- \mathbb{E} T_n \geq t) \nonumber \\
	&\leq \inf_{\lambda \in (0,n\alpha/2^{d+1})} \exp \Bigl( - \lambda t + \frac{5 \lambda^2}{32n} - n \log \Bigl( 1 - \frac{ 2^{2d+2} \lambda^2}{n^2 \alpha^2} \Bigr) - n \log \Bigl( 1 - \frac{2^{2d} \lambda^2}{n^2 \alpha^2} \Bigr) \Bigr) \nonumber \\
	& \leq \exp \Bigl( - \frac{n \alpha^2 t^2}{2^{2d+4}} + \frac{5 n \alpha^4 t^2}{2^{4d+9}} - n \log \Bigl( 1 - \frac{ \alpha^2 t^2}{2^{2d+6}} \Bigr) - n \log \Bigl(1 - \frac{\alpha^2 t^2}{2^{2d+8}} \Bigr) \Bigr) \nonumber \\
	& \leq \exp \Bigl( - \frac{n \alpha^2 t^2}{2^{2d+6} } \Bigr),
\end{align}
and similarly that the same bound applies to $\mathbb{P}(T_n - \mathbb{E} T_n \leq -t)$. 

Now, writing $\mathcal{X}_{x_0}^*:= \{x \in \mathbb{R}^d : \|x- x_{j^*(x_0)}\|_\infty < h\}$,  we have that
\[
	\mathbb{E}(T_n) = \mathbb{E}\Bigl( W_1 - \frac{1}{2} U_1 \Bigr) = \mathbb{E} \bigl[ \mathbbm{1}_{\{X \in \mathcal{X}_{x_0}^*\}} \{ \eta(X)-1/2\} \bigr].
\]
Define the set
\[
	\Omega_0 := \biggl\{ x_0 \in \mathbb{R}^d : \liminf_{h \rightarrow 0} \frac{\mathbb{P}( \|X-x_0\|_\infty <h )}{h^d} > 0 \biggr\}.
\]
and note that, for $x_0 \in \Omega_0$ we have that $n \alpha^2 \mathbb{P}(X \in \mathcal{X}_{x_0}^*)^2 \geq n \alpha^2 \mathbb{P}(\|X-x_0\|_\infty \leq h/2)^2 \rightarrow \infty$, since we are assuming that $n \alpha^2 h^{2d} \rightarrow \infty$.
We will now establish that $\mathbb{P}(X \in \Omega_0)=1$. By the Lebesgue decomposition theorem \citep[e.g. Theorem 6.10(a),][]{Rudin1987}, there exist measures $\mu,\nu$ on $\mathbb{R}^d$ such that $\mu$ is absolutely continous with respect to Lebesgue measure and $\nu$ is singular with respect to Lebesgue measure, and such that
\[
	\mathbb{P}(X \in A) = \mu(A) + \nu(A)
\]
for all measurable $A \subseteq \mathbb{R}^d$. Define the sets
\[
	\Omega_\mu := \biggl\{ x_0 \in \mathbb{R}^d : \liminf_{h \rightarrow 0} \frac{\mu(B_{x_0}(h))}{h^d} > 0 \biggr\}
\]
and
\[
	\Omega_\nu := \biggl\{ x_0 \in \mathbb{R}^d : \liminf_{h \rightarrow 0} \frac{\nu(B_{x_0}(h))}{h^d} > 0 \biggr\}	
\]
so that $\Omega_0 = \Omega_\mu \cup \Omega_\nu$ and hence that $\mathbb{P}(X \not\in \Omega_0) \leq \mu(\Omega_\mu^c) + \nu( \Omega_\nu^c)$. Since $\mu$ is absolutely continuous, it has a density $f_\mu$ and we have that $\Omega_\mu= \{ x_0 \in \mathbb{R}^d : f_\mu(x_0)>0\}$ and $\mu(\Omega_\mu^c)=0$.

We now turn to the singular part $\nu$ and write $\lambda$ for Lebesgue measure. Since $\nu$ is singular there exists a measurable set $A \subseteq \mathbb{R}^d$ such that $\nu(A^c)=0$ and $\lambda(A)=0$. Fixing an arbitrary $\epsilon>0$ we can always find an open set $O_\epsilon$ such that $A \subseteq O_\epsilon$ and $\lambda(O_\epsilon) < \epsilon$. Now for each $x \in \Omega_\nu^c \cap A$ we may choose $r_x \in (0,1)$ such that $B_x(r_x) \subseteq O_\epsilon$ and $\nu(B_x(5r_x)) \leq \lambda(B_x(5r_x)).$ By the Vitali covering lemma \citep[e.g. Theorem 1.24,][]{EvansGariepy2015} there exists a countable subset $\mathcal{I} \subset \Omega_\nu^c \cap A$ such that the balls $( B_x(r_x) )_{x \in \mathcal{I}}$ are disjoint and such that $\Omega_\nu^c \cap A \subseteq \bigcup_{x \in \mathcal{I}} B_x(5r_x)$. Hence
\begin{align*}
	\nu(\Omega_\nu^c) = \nu( \Omega_\nu^c \cap A) \leq \sum_{x \in \mathcal{I}} \nu( B_x(5r_x)) &\leq \sum_{x \in \mathcal{I}} \lambda( B_x(5r_x)) \\
	&\leq 5^d \sum_{x \in \mathcal{I}} \lambda( B_x(r_x)) \leq 5^d \lambda(O_\epsilon) < 5^d \epsilon.
\end{align*}
Since $\epsilon>0$ was arbitrary, we have that $\nu(\Omega_\nu^c) =0$ and hence that $\mathbb{P}(X \in \Omega_0)=1$.

For $x_0 \in \Omega_0$ we have that $\mathbb{P}(X \in \mathcal{X}_{x_0}^*) >0$ and we may consider the event
\[
	\Omega_1:= \biggl\{ x_0 \in \Omega_0 : \limsup_{h \rightarrow 0} \biggl| \frac{\mathbb{E}(T_n)}{\mathbb{P}(X \in \mathcal{X}_{x_0}^*)} - \{ \eta(x_0) -1/2 \} \biggr| =0 \biggr\}.
\]
By the Lebesgue differentiation theorem \citep[e.g. Lemma 4.1.2,][]{LedrappierYoung1985} we have that $\mathbb{P}(X \in \Omega_1)=1$.  For $x_0 \in \Omega_1$ such that $\eta(x_0) \neq 1/2$ we now have, writing $\delta = \frac{1}{2} |\eta(x_0) - 1/2| \mathbb{P}(X \in \mathcal{X}_{x_0}^*)$, that
\begin{equation}
\label{Eq:Sign}
	\delta^{-1} \mathbb{E} T_n(x_0) \rightarrow 2 \text{ sign}(2\eta(x_0)-1).
\end{equation}
Moreover, for $x_0 \in \Omega_1$, using~\eqref{Eq:Subgaussian} we also have
\begin{align*}
	\bigl\{ \mathbb{P}(&T_n(x_0) < 0 ) - \mathbbm{1}_{\{\eta(x_0) < 1/2\}} \bigr\} \{2 \eta(x_0) -1 \} \bigl( \mathbbm{1}_{\{\mathbb{E} T_n \geq \delta, 2\eta(x_0) > 1\}} +  \mathbbm{1}_{\{\mathbb{E} T_n \leq -\delta, 2\eta(x_0) < 1\}} \bigr) \\
	& \leq \mathbb{P}( T_n(x_0) <0) \mathbbm{1}_{\{\mathbb{E} T_n \geq \delta, 2\eta(x_0) > 1\}} + \mathbb{P}( T_n(x_0) \geq 0) \mathbbm{1}_{\{\mathbb{E} T_n \leq -\delta, 2\eta(x_0) < 1\}} \\
	& \leq \mathbb{P}( |T_n(x_0) - \mathbb{E} T_n(x_0) | \geq \delta) \mathbbm{1}_{\{2\eta(x_0) \neq 1\}} \leq 2 \exp \Bigl( - \frac{n \alpha^2 \delta^2}{2^{2d+6}} \Bigr) \mathbbm{1}_{\{2\eta(x_0) \neq 1\}} \rightarrow 0,
\end{align*}
where for the final convergence statement we have used the facts that $x_0 \in \Omega_0$ and $n \alpha^2 h^{2d} \rightarrow \infty$. It now follows from this and from~\eqref{Eq:Sign} that for $x_0 \in \Omega_1$ we have
\begin{align*}
	\bigl\{ \mathbb{P}(T_n(x_0) &< 0 ) - \mathbbm{1}_{\{\eta(x_0) < 1/2\}} \bigr\} \{2 \eta(x_0) -1 \} \\
	& \leq 2\exp\Bigl( - \frac{n \alpha^2 \delta^2}{2^{2d+6}} \Bigr) \mathbbm{1}_{\{2\eta(x_0) \neq 1\}} + \mathbbm{1}_{\{\mathbb{E} T_n \geq \delta, 2\eta(x_0) < 1\}} \\
	& \hspace{150pt}+ \mathbbm{1}_{\{\mathbb{E} T_n \leq -\delta, 2\eta(x_0) > 1 \}}  + \mathbbm{1}_{\{|\mathbb{E} T_n| < \delta, 2\eta(x_0) \neq 1\}} \\
	& \leq 2\exp\Bigl( - \frac{n \alpha^2 \delta^2}{2^{2d+6}} \Bigr) \mathbbm{1}_{\{2\eta(x_0) \neq 1\}} + 2 \mathbbm{1}_{\{ \frac{\mathbb{E} T_n(x_0)}{ \delta \text{sign}(2 \eta(x_0) -1 )} \leq -1, 2 \eta(x_0) \neq 1\}}\\
	& \hspace{150pt} + \mathbbm{1}_{\{ | \mathbb{E} T_n(x_0) / \delta | < 1, 2\eta(x_0) \neq 1 \}} \rightarrow 0.
\end{align*}
Hence, by the dominated convergence theorem \citep[e.g. Theorem 1.34,][]{Rudin1987}, we have that
\[
	\mathcal{E}_P(C_n) = \mathbb{E} \Bigl[ \bigl\{ \mathbb{P}(T_n(X) < 0 | X ) - \mathbbm{1}_{\{\eta(X) < 1/2\}} \bigr\} \{2 \eta(X) -1 \} \Bigr] \rightarrow 0,
\]
as required.
\end{proof}
\begin{proof}[Proof of Theorem~\ref{Prop:UpperBound}]
We adopt the same notation as in the proof of Theorem~\ref{Thm:UniversalConsistency} and recall that
\[
	\mathbb{E}(T_n) = \mathbb{E}\Bigl( W_1 - \frac{1}{2} U_1 \Bigr) = \int_{\mathcal{X}_{x_0}^*} f(x)\{\eta(x)-1/2\} \,dx.
\]
Now, if $x$ belongs to $B_{x_0}(h/2)$ and given that $\|x_0- x_{j^*(x_0)}\|_\infty < h/2$, we get that $\|x-x_{j^*(x_0)}\|_\infty < h$. This implies that $B_{x_0}(h/2) \subseteq \mathcal{X}_{x_0}^*$, so $\mathbb{P}(X_1 \in \mathcal{X}_{x_0}^*) \geq c_0 \mu(h/2)^d$ whenever $x_0 \in \mathrm{supp}(f)$ and $h/2 \leq r_0$. Thus, if $\eta(x_0) \geq 1/2 + L(2h d^{1/2})^\beta$ we therefore have that
\begin{align}
\label{Eq:Bias}
	\mathbb{E}(T_n) \geq \int_{\mathcal{X}_{x_0}^*} f(x) \{ \eta(x_0) - 1/2& - L(2hd^{1/2})^{\beta} \} \,dx \nonumber \\
	&\geq c_0 \mu (h/2)^d \{\eta(x_0) - 1/2 - L(2hd^{1/2})^{\beta} \},
\end{align}
and similarly if $\eta(x_0) \leq 1/2 - L(2hd^{1/2})^\beta$ we have $\mathbb{E}(T_n) \leq - c_0 \mu (h/2)^d \{1/2 - \eta(x_0) - L(2hd^{1/2})^\beta \}$. We can now bound the excess risk of our classifier by appropriately partitioning $\{ x \in [0,1]^d : \eta(x) \neq 1/2\}$ according to the size of $|\eta(x)-1/2|$. To this end, define
\[
	\mathcal{X}_0:= \bigl\{ x \in [0,1]^d : 0< |\eta(x) - 1/2| \leq 2L(2hd^{1/2})^{\beta} \bigr\}
\]
and, for each $j=1,2,\ldots$,
\[
	\mathcal{X}_j:= \bigr\{ x \in [0,1]^d : 2^j L(2hd^{1/2})^{\beta} <  |\eta(x) - 1/2| \leq 2^{j+1} L(2hd^{1/2})^{\beta} \bigr\}.
\]
Combining~\eqref{Eq:Subgaussian} and~\eqref{Eq:Bias} we can now see that
\begin{align*}
	&\mathcal{E}_P(C_n) = \int_{\mathrm{supp}(f)} f(x_0) \bigl\{ \mathbb{P}(T_n(x_0) < 0 ) - \mathbbm{1}_{\{\eta(x_0) < 1/2\}} \bigr\} \{2 \eta(x_0) -1 \} \,dx_0 \\
	& \leq 4 L (2hd^{1/2})^\beta \mathbb{P}( X \in \mathcal{X}_0) \\
	&\hspace{50pt} + \sum_{j=1}^\infty \int_{\mathcal{X}_j} f(x_0) \{2 \eta(x_0) -1 \} \exp \Bigl( - \frac{n \alpha^2}{2^{2d+6}} [ \mathbb{E}\{T_n(x_0)\} ]^2 \Bigr) \,dx_0 \\
	& \leq 4 L (2hd^{1/2})^\beta \mathbb{P}( X \in \mathcal{X}_0) \\
	&\hspace{50pt}+ \sum_{j=1}^\infty 2^{j+2} L (2hd^{1/2})^\beta \mathbb{P}(X \in \mathcal{X}_j) \exp \Bigl( - \frac{n \alpha^2}{2^{4d+6-2\beta}} d^\beta c_0^2 \mu^2 L^2 h^{2d+2\beta} 2^{2j} \Bigr).
\end{align*}
This motivates the choice of bandwidth $h = (n \alpha^2 L^2 \mu^2)^{-1/(2d + 2\beta)}$. With this choice, there exists $A = A(d,\beta) >0$ such that
\begin{align*}
	\mathcal{E}_P(C_n) \leq \frac{A}{2^{\beta(1+\gamma)}} C_0 (Lh^\beta)^{1+\gamma} \leq A C_0 \biggl\{ \frac{L^d}{(n \alpha^2 \mu^2)^{\beta/2}} \biggr\}^\frac{1+\gamma}{d+\beta}.
\end{align*}
\end{proof}
\begin{proof}[Proof of Theorem~\ref{Prop:LowerBound}]
This result is proved by applying Theorem~3 of~\citet{DJW18}, which provides a private version of Assouad's lemma, to the construction used to establish the lower bound in~\citet{AudibertTsybakov07}. For $q \in \mathbb{N}$ define $\mathcal{J}:=\{0,1,\ldots,q\}^d$ and consider the grid
\[
	G_q:=\biggl\{ x_j= \biggl( \frac{2j_1+1}{4q}, \ldots, \frac{2j_d+1}{4q} \biggr) : j \in \mathcal{J} \biggr\}.
\]
Write $n_q(x):=\mathrm{argmin}_{x' \in G_q} \|x-x'\|$ and for each $j \in \mathcal{J}$ let $\tilde{\mathcal{X}}_j=\{x \in [0,1/2]^d : j^*(x) = j \}$. For an integer $m \leq q^d$ let $\{j^{(1)},\ldots,j^{(m)}\} \subseteq \mathcal{J}$ be any subset of size $m$, write $\mathcal{X}_i:=\tilde{\mathcal{X}}_{j^{(i)}}$ for $i=1,\ldots,m$ and write $\mathcal{X}_0:= [0,1]^d \setminus \bigcup_{i=1}^m \mathcal{X}_i$ so that $\mathcal{X}_0, \mathcal{X}_1, \ldots, \mathcal{X}_m$ forms a partition of $[0,1]^d$. Further, let $u: [0,\infty) \rightarrow [0,\infty)$ be nonincreasing and smooth such that
\[
	u(x)=1 \text{ for } x \in [0,1/8] \quad \text{ and } \quad u(x)=0 \text{ for } x \in [1/4, \infty).
\]
Define $\phi(x)=L C_\phi u(\|x\|)$ with $C_\phi$ chosen small enough that $\phi$ is $(\beta,L)$-H\"older smooth.

For each $\sigma \in \{-1,+1\}^m$ we now define a distribution $P_\sigma$ of $(X,Y)$. Each of these distributions has the same marginal distribution which, for some $A_0 \subset \mathcal{X}_0$ sufficiently regular with positive Lebesgue measure and $w \in (0,1/m]$, has a density given by
\[
	f(x):= \left\{ \begin{array}{lr} (8q)^dw/V_d & \text{if } x \in B_{x_{j^{(i)}}}(1/(8q)) \text{ for some } i=1,\ldots, m \\ (1-mw)/\lambda(A_0) & \text{if } x \in A_0 \\ 0 & \text{otherwise} \end{array} \right.,
\]
where we write $V_d:=\lambda(B_0(1))=\frac{\pi^{d/2}}{\Gamma(1+d/2)}$. We define the conditional distribution of $X|Y$ by writing $\varphi(x) = q^{-\beta} \phi(q(x-n_q(x)))$, $\eta_\sigma(x):=1/2$ when $x \in \mathcal{X}_0$ and
\[
	\eta_\sigma(x):= \frac{1+ \sigma_i \varphi(x)}{2}
\]
when $x \in \mathcal{X}_i$ for $i=1,\ldots,m$. With $C_\phi$ chosen sufficiently small we have that $\eta_\sigma$ is $(\beta,L)$-H\"older smooth for all values of $\sigma \in \{-1,+1\}^m$. Moreover
\begin{align*}
	\mathbb{P}_\sigma(0 < |\eta_\sigma(X) -1/2| \leq t) = m \mathbb{P}_\sigma(0 < \phi(q(X-x_{j^{(1)}})) \leq 2t q^{\beta}) =  mw \mathbbm{1}_{\{ 2tq^\beta \geq L C_\phi \}},
\end{align*}
so that the $(\gamma,C_0)$-margin condition is satisfied when $mw \leq C_0 ( \frac{L C_\phi}{2 q^\beta})^{\gamma}$.

Now let $C_n$ be any classifier only depending on $Z_1,\ldots,Z_n$ and let $Q \in \mathcal{Q}_\alpha$ be any privacy mechanism that satisfies $\alpha$-LDP. If $(X,Y) \sim P_\sigma$ then write $M_\sigma^{(n)}$ for the distribution of the privatised data $(Z_1,\ldots,Z_n)$. For any $\sigma \in \{-1,+1\}^m$, $k=1,\ldots,m$ and $r \in \{-1,0,+1\}$ we will write $\sigma_{k,r}:=( \sigma_1, \ldots, \sigma_{k-1}, r, \sigma_{k+1}, \ldots, \sigma_m)$ and we will write $L_{k,r}$ for the likelihood ratio of $M_{\sigma_{k,r}}^{(n)}$ to $M_{\sigma_{k,0}}^{(n)}$. Then
\begin{align*}
	&\sup_{P \in \mathcal{P}(\theta)} \mathcal{E}_P(C_n) \geq 2^{-m} \sum_{\sigma \in \{-1,+1\}^m} \mathcal{E}_{P_\sigma}(C_n) \\
	&= 2^{-m} \sum_{\sigma \in \{-1,+1\}^m} \sum_{k=1}^m \sigma_k \int_{\mathcal{X}_k} \{ \mathbb{P}_\sigma(C_n(x)=0) - \mathbbm{1}_{\{ \sigma_k<0 \}} \} \varphi(x) f(x) \,dx \\
	& \geq 2^{-m} \sum_{\substack{ \sigma \in \{-1,+1\}^m : \\ \sigma_k=1}} \sum_{k=1}^m \int_{\mathcal{X}_k} \mathbb{E}_{\sigma_{k,0}} \bigl\{ L_{k,1}\mathbbm{1}_{\{C_n(x)=0\}} + L_{k,-1} \mathbbm{1}_{\{C_n(x)=1\}} \bigr\}  \varphi(x) f(x) \, dx \\
	& \geq 2^{-m} \sum_{\substack{ \sigma \in \{-1,+1\}^m : \\ \sigma_k=1}} \sum_{k=1}^m \int_{\mathcal{X}_k} \mathbb{E}_{\sigma_{k,0}} \{ \min(L_{k,1}, L_{k,-1}) \} \varphi(x) f(x) \, dx \\
	&= \frac{1}{2} \sum_{k=1}^m \{ 1- d_\mathrm{TV}(M_{+k}^{(n)}, M_{-k})^{(n)}\} \int_{\mathcal{X}_k} \varphi(x) f(x) \, dx,
\end{align*}
where we write $M_{rk}^{(n)} =2^{-(m-1)} \sum_{\sigma : \sigma_k=r} M_\sigma^{(n)}$. With our choice of $\varphi$ we have that
\[
	\int_{\mathcal{X}_k} \varphi(x) f(x) \,dx =  \frac{(8q)^d w}{V_d} \int_{B_0(1/(8q))} q^{-\beta} \phi(qx) \,dx = w L C_\phi q^{-\beta}.
\]
Now write $P_{rk}=2^{-(m-1)} \sum_{\sigma: \sigma_k=r} P_\sigma$. Using Cauchy--Schwarz, Pinsker's inequality and Theorem~3 of \citet{DJW18} we can see that
\begin{align*}
	\biggl\{ \frac{1}{m} \sum_{k=1}^m d_\mathrm{TV}&(M_{+k}^{(n)}, M_{-k}^{(n)}) \biggr\}^2  \leq  \frac{1}{4m} \sum_{k=1}^m \{ d_\mathrm{KL}(M_{+k}^{(n)},M_{-k}^{(n)}) + d_\mathrm{KL}(M_{-k}^{(n)},M_{+k}^{(n)})\}  \\
	& \leq \frac{2n(e^\alpha-1)^2}{m} \sum_{k=1}^m d_\mathrm{TV}(P_{+k}, P_{-k})^2 = 2n(e^\alpha-1)^2 w^2L^2C_\phi^2 q^{-2\beta}.
\end{align*}

Take $q=\lceil C_1 (n\alpha^2)^{1/(2\beta+2d)} \rceil$, $w=c_2 q^{-d}$ and $m= \lfloor C_3 q^{d-\beta \gamma} \rfloor$ for some constants $C_1,c_2,C_3$. If $C_1$ and $C_3$ are chosen large enough then we have $q \geq 1$ and $m \geq 2$. Moreover, $mw \leq c_2 C_3 q^{-\alpha \beta}$ so, for $c_2$ small enough we have both $w \in (0,1/m]$ and $mw \leq C_0(\frac{LC_\phi}{2 q^\beta})^\gamma$. Furthermore, if $c_2$ is chosen small enough then we have $2n(e^\alpha-1)^2 w^2 L^2 C_\phi^2 q^{-2\beta} \leq 8 n \alpha^2 c_2^2 L^2 C_\phi^2 q^{-2\beta-2d} \leq 8 c_2^2 L^2 C_\phi^2 \leq 1/2$. For these choices of $q,w,m$ we then have
\begin{align*}
	\sup_{P \in \mathcal{P}(\theta)} \mathcal{E}_P(C_n) \geq \frac{mwLC_\phi }{4q^\beta} &\geq \frac{C_3 c_2  L C_\phi}{8q^{\beta(1+\gamma)}} \geq 2^{-\beta(1+\gamma)-3} C_3 c_2 L C_\phi (n \alpha^2)^{-\frac{\beta(1+\gamma)}{2 \beta + 2d}}.
\end{align*}
\end{proof}

\end{document}